\documentclass[12pt]{amsart}
\usepackage{amsbsy,amssymb,amscd,amsfonts,latexsym,amstext,delarray,
amsmath,graphicx, mathrsfs} 
\input xypic

\usepackage{tikz-cd}

\usepackage{color}

\newtheorem{thm}{Theorem}[section]
\newtheorem{prop}[thm]{Proposition}

\newtheorem{lem}[thm]{Lemma}

\newtheorem{defn}[thm]{Definition}
\newtheorem{rem}[thm]{Remark}

\numberwithin{equation}{section}

\def\bL{{\mathbb L}}

\def\bP{{\mathbb P}}

\def\bR{{\mathbb R}}

\def\N{{\mathbb N}}

\def\cA{{\mathcal A}}

\def\cC{{\mathcal C}}
\def\cD{{\mathcal D}}

\def\cM{{\mathcal M}}
\def\cN{{\mathcal N}}

\def\cT{{\mathcal T}}

\def\cW{{\mathcal W}}

\def\Hom{{\rm Hom}}
\def\id{{\rm id}}

\def\cof{{\textbf{cof}}}
\def\llp{{\textbf{llp}}}
\def\rlp{{\textbf{rlp}}}
\def\cell{{\textbf{cell}}}
\def\fib{{\textbf{fib}}}
\def\Mor{{\text{Mor}}}
\def\Cat{{\textbf{Cat}}}
\def\PrSh{{\textbf{PrSh}}}
\def\Flow{{\textbf{Flow}}}
\def\sSemiCat{{\textbf{sSemiCat}}}

\title[Comparison Between Different Topological Models of Concurrency]{Comparison Between Different Topological Models of Concurrency} 
\author{Joshua F.~Lieber}

\address{California Institute of Technology, Pasadena \\ USA}
\email{jlieber@caltech.edu}

\begin{document}
\maketitle

\begin{abstract}
    In this note, we provide an explicit non-Quillen equivalence between the category of precubical sets and Gaucher's category of flows via a class of "realization functors" (with mild assumptions on the cofibrations of the category of precubical sets).  In addition, we demonstrate a Quillen equivalence between simplicial semicategories and flows before proving that simplicial semicategories satisfy many of the same properties as flows.  Finally, we introduce the category of boxed symmetric trees, presheaves on which may provide a slightly more flexible setting for concurrent computing than (pre)cubical sets, before showing that when endowed with degeneracies, the aforementioned presheaf category is a test category (although not strict test).  
\end{abstract}

\section{Introduction}

Over the years, numerous models for concurrent computing have been proposed, each with unique advantages and disadvantages.  One hope is that at least some of these models might be equivalent in a suitably weak sense, so that in choosing to work with one model over another, one is not really making a choice at all (with respect to all relevant data).  Unfortunately, this does not seem to be the case in general.  

The two existing models that we consider in this article are precubical sets (so-called higher dimensional automata, c.f. \cite{FGHMR}) and flows (first investigated by Gaucher in \cite{Gau1}).  Precubical sets are a relatively classical model of concurrent computing.  In a given cube, each of the different directions correspond to different concurrently executing computations.  One can think of each direction in a standard n-cube as coding for a factor of "done-ness" which meets with all other possible executions at the other end of the longest diagonal.  More complicated computations are merely built out of these basic units.  On the other hand, the category of flows basically consists of topologically enriched semicategories (categories without identity) and semifunctors between them (continuous on Hom spaces).  Each object of a given flow can be seen as a different state that a computation can be in, and each morphism of that flow can be seen as an execution path between different states.  The reason for topologizing the spaces of morphisms is to allow us to have a notion of (computational) equivalence between different execution paths, and to allow us to differentiate between different ways of moving between paths between execution paths, and so on.

The structure of our article is as follows.  In section 2, we provide a rapid refresher on many of the central notions of the theory of model categories.  This includes the basics on model structures and homotopy categories, as well as Quillen adjunction and Quillen equivalence.  We also introduce Cofibrantly Generated and Combinatorial model categories, which provide most of the basic setting in which we work throughout the article (every model structure we define will end up being combinatorial).  The one more advanced topic we consider will be that of left-induction of model structure, whereby one pulls back a model structure along a left adjoint.  This seemingly obvious notion is actually rather non-trivial, as several technical conditions must be satisfied for left-induction to succeed (thankfully, all but one of these may be elided by virtue of the fact that our model structures are combinatorial).  

In section 3, we begin by introducing the model category of flows.  After discussing a few basic properties of this category (largely citing Gaucher's work itself for the proofs), we provide the definition (also due to Gaucher) of a class of functors from the category of precubical sets to flows called realization functors.  These functors can basically be thought of as semicategorical (and topological) analogues of the realization functor from simplicial sets to simplicially enriched categories.  Analogously, these realizations also admit right adjoint nerve functors.  This is all a lead-up to our first theorem of the paper, namely:

\begin{thm}
Suppose that there is a model structure on the category of precubical sets along with a realization functor $L:\textbf{PrSh}(\square)\rightleftarrows\textbf{Flow}$ which is the left Quillen adjoint in a Quillen pair $(L\dashv N_{L}):\textbf{PrSh}(\square)\rightleftarrows\textbf{Flow}$ (note that we are essentially only assuming that everything in $I=\{\partial\square[n]\hookrightarrow\square[n]\}_{n=0}^\infty\cup\{\square[0]\sqcup\square[0]\to\square[0]\}$ is sent to a cofibration and that the images of $\square[n]$ are weakly equivalent to $\{0<1\}^n$ for any $n$).  Then this adjunction cannot be a Quillen equivalence.
\end{thm}

In spite of this, there is a reasonable combinatorial model for flows, provided by simplicial semicategories (with simplicial semifunctors as morphisms).  We begin the section by introducing a few key definitions, and then set up a geometric realization/nerve adjunction pair (essentially just applying the geometric realization/singular space adjunction between simplicial sets and topological spaces on the Homs of simplicial semicategories or flows).  This allows the model structure to be left induced.

\begin{prop}
The model structure on $\Flow$ may be left induced via the adjunction $(|-|\dashv Sing):\sSemiCat\rightleftarrows\Flow$ (constant on objects and acting via realization/singular set on Hom objects).  This upgrades $(|-|\dashv Sing)$ into a Quillen pair.
\end{prop}

This then brings us to our second main theorem of the section.

\begin{thm}
The Quillen adjunction $(|-|\dashv Sing):\sSemiCat\rightleftarrows\Flow$ is a Quillen equivalence.
\end{thm}

We then show that there is an equivalent model structure on $\sSemiCat$ that has a nice set of generating cofibrations.  Afterwards, we go on to demonstrate that $\sSemiCat$ has several properties in common with the category of flows (indeed, many of the proofs become even simpler in $\sSemiCat$), including a way of defining a stronger notion homotopy equivalence.

In Section 3, we change gears, and define a small category called the category of boxed-symmetric trees, denoted $\boxed{\cT}$.  One can think of this category as an analogue of the category of cubes, where we allow any rooted tree (either all directed away from the root or all directed towards) as a "basic interval."  Presheaves on this allow us a bit more flexibility as models of concurrency, as our basic "cubes" in this case correspond to performing flowchart computations in any of several concurrent directions.  As a quick check, we prove that $\boxed{\cT}$ is a test category in the sense of Grothendieck.

\begin{thm}
$\boxed{\cT}$ is a test category.
\end{thm}

\bigskip

\textbf{Convention:} In this note, $\textbf{Top} $ will refer to a convenient category of topological spaces (such as compactly generated weak Hausdorff spaces or $\Delta$-generated spaces), i.e., a full replete subcategory of the category of all topological spaces which is cartesian-closed, bicomplete, and which contains all CW complexes.  Topological space will be used to mean a member of this convenient category.

\section{Model Categories}

The following section is devoted to introducing enough of the basic definitions of the theory of model categories (a certain type of category equipped with additional structure that provides a "good setting for homotopy theory" originally defined by Dan Quillen in \cite{Qui}) to understand the following sections.  In particular, this should not be thought of as a comprehensive introduction to the theory of model categories, and there will be several glaring omissions even in the very basics.  For a good introduction to the extremely rich theory of model categories, we refer the interested reader to \cite{Hir}.

Classical sources of motivation for model categories come from considering topological spaces/simplicial sets up to weak homotopy equivalence, and the Gabriel-Zisman localization of a category.  Given a category $\cC$ and any class of morphisms $\cW\subset\Mor(\cC)$, we may form its \emph{Gabriel-Zisman localization} $\cC[\cW^{-1}]$ by formally inverting the morphisms in $\cW$.  In general, this is extremely poorly behaved.  For example, if one begins with a locally small category, its localization at a class of morphisms need not be locally small in general (and, in fact, often isn't).  As will be noted later, model categories provide one setting in which the localization can be controlled (in the sense that it will be equivalent to a category with a much simpler description\textemdash one which, thankfully, is always locally small when one starts with a locally small category).

\subsection{Model Structures, Model Categories, and Homotopy Categories}

\begin{defn}
Suppose we have a category $\cC$.  A \emph{model structure} on $\cC$ consists of three classes of maps, \emph{weak equivalences} $\cW_\cC$, \emph{cofibrations} $\cof_\cC$, and \emph{fibrations} $\fib_\cC$ (we will suppress the subscripts if the context is clear) satisfying the following axioms.
\begin{itemize}
    \item (2-out-of-3 axiom) Given morphisms $f,g\in\textbf{Mor}(\cC)$ such that $g\circ f$ is defined, if any two of $f$, $g$, and $g\circ f$ are in $\cW_\cC$, then so is the third.
    \item (retract axiom) If $f$ is a retract of $g$, and $g$ is a weak equivalence, fibration, or cofibration, then so is $f$.
    \item (lifting axiom) suppose we have the commutative diagram of solid arrows
    \[
    \begin{tikzcd}
  A\arrow[d, "i"]\arrow[r]& X\arrow[d, "p"]\\
  B\arrow[r]\arrow[ru, dashed]& Y
\end{tikzcd}.
    \]
    Then the dotted arrow exists and results in a commutative diagram if
    \begin{enumerate}
        \item $i\in\cof_\cC$ and $p\in\fib_\cC\cap\cW_\cC$
        \item $i\in\cof_\cC\cap\cW_\cC$ and $p\in\fib_\cC$
    \end{enumerate}
    \item (factorization axiom)  There are two functorial factorizations of every morphism $f\in\textbf{Mor}(\cC)$.
    \begin{enumerate}
        \item $f=qi$, where $q\in\fib_\cC\cap\cW_\cC$ and $i\in\cof_\cC$
        \item $f=pj$, where $p\in\fib_\cC$ and $j\in\cof_\cC\cap\cW_\cC$.
    \end{enumerate}
\end{itemize}
Elements of $\fib_\cC\cap\cW_\cC$ are known as \emph{trivial fibrations} and elements of $\cof_\cC\cap\cW_\cC$ are known as \emph{trivial cofibrations}
\end{defn}

\begin{defn}
A bicomplete category $\cC$ equipped with a model structure is known as a \emph{model category}.
\end{defn}

Now, we wish to demonstrate that homotopy categories have well-behaved localizations with respect to their classes of weak equivalences.  This will involve several definitions.

\begin{defn}
Given a model category $\cM$, and object $X\in\cM$ will be known as \emph{fibrant} if the unique map $X\to*$ is a fibration.  Analagously, $X$ will be known as \emph{cofibrant} if the unique map $\emptyset\to X$ is a cofibration.
\end{defn}

Note that we have canonical functors $\cM\to\cM^{\Delta^1}$ given by $X\mapsto(X\to*)$ and $X\mapsto(\emptyset\to X)$.  By our assumptions above, we have a functorial factorization of $X\to*$ into $X\to X^{fib}\to*$, where $X\to X^{fib}$ is a trivial cofibration and $X^{fib}\to*$ is a fibration. Analogously, we have a functorial factorization of $\emptyset\to X$ into $\emptyset\to X^{cof}\to X$, where $\emptyset\to X^{cof}$ is a cofibration and $X^{cof}\to X$ is a trivial fibration.

\begin{defn}
The endofunctors $(-)^{fib}$ and $(-)^{cof}$ of $\cM$ we implicitly defined in the preceding paragraph are known as \emph{fibrant replacement} and \emph{cofibrant replacement}, respectively.
\end{defn}

Recalling that our model category admits small coproducts, for any object $X\in\cM$, we may define the fold map $X\sqcup X\to X$.  By functorial factorization once again, this allows us to define a \emph{good cylinder object} for $X$ by factoring $X\sqcup X\to X$ as 
\[
X\sqcup X\overset{i_0\sqcup i_1}{\longrightarrow}\text{Cyl}(X)\overset{p}{\to}X,
\]
where $i_0\sqcup i_1$ is a cofibration and $p$ is a trivial fibration.

\begin{defn}
Two morphisms $f, g:X\to Y$ are known as \emph{left homotopy equivalent} if there exists a map $H:\text{Cyl}(X)\to Y$ such that the compositions $H\circ i_0=f$ and $H\circ i_1=g$.
\end{defn}

As one might suspect from the name, left homotopy equivalence generates an equivalence relation on the Hom-sets of $\cM$.  In fact, we have the following definition and theorem.

\begin{defn}
Given a model category $\cM$, one defines its \emph{homotopy category} $\text{Ho}(\cM)$ as follows:
\begin{itemize}
    \item The objects of $\text{Ho}(\cM)$ are the objects of $\cM$ which are both fibrant and cofibrant
    \item Given any two $X, Y\in\text{Ho}(\cM)$, one has that $\Hom_{\text{Ho}(\cM)}(X, Y)$ is the quotient of $\Hom_\cM(X, Y)$ by left homotopy equivalence.
\end{itemize}
\end{defn}

\begin{thm}
For any model category $\cM$, its homotopy category $\text{Ho}(\cM)$ is equivalent to its localization $\cM[\cW_\cM^{-1}]$ by its class of weak equivalences.
\end{thm}

\begin{proof}
This is \cite{Hir} Theorem 8.3.9.
\end{proof}

Indeed, if our starting model category $\cM$ was locally small, $\text{Ho}(\cM)$ is a locally small model for $\cM[\cW_\cM^{-1}]$.

\subsection{Quillen Adjunction and Quillen Equivalence}

Now that we have defined the notion of a model category, it would be helpful to have some way of comparing the model structures on these categories.  Note that arguably the most important form of comparison between two categories (equipped with no extra structure) is the data of an adjunction between the two.  With that in mind, we have the following definition.

\begin{defn}
Given two model categories $\cM$ and $\cN$, a \emph{Quillen adjunction} or \emph{Quillen pair} between $\cM$ and $\cN$ is the data of an adjunction $(L\dashv R):\cM\rightleftarrows\cN$ between the two such that one of the following equivalent conditions is satisfied:
\begin{enumerate}
    \item $L$ preserves cofibrations and trivial cofibrations
    \item $R$ preserves fibrations and trivial fibrations
    \item $L$ preserves cofibrations and $R$ preserves fibrations
    \item $L$ preserves trivial cofibrations and $R$ preserves trivial fibrations
\end{enumerate}
\end{defn}

One particularly important aspect of the notion of a Quillen adjunction is that it induces an adjunction on the level of homotopy theory.  In other words,

\begin{thm}
Given a Quillen adjunction $(L\dashv R):\cM\rightleftarrows\cN$, it induces an adjunction $(\bL\dashv \bR):\text{Ho}(\cM)\rightleftarrows\text{Ho}(\cN)$ between homotopy categories.
\end{thm}

\begin{proof}
This can be found in \cite{Hir} (several propositions in section 8.5).
\end{proof}

This gives us a particularly well-structured way of comparing two homotopy theories.  Of course, the nicest form of adjunction between categories is a categorical equivalence.  It will be especially useful to us to import the notion of equivalence into this weaker setting.  

\begin{defn}
Given a Quillen adjunction $(L\dashv R):\cM\rightleftarrows\cN$, we say that it is a \emph{Quillen equivalence} if it descends to an equivalence of categories on the level of homotopy theory.  In particular, this holds if it satisfies one of the following conditions:
\begin{itemize}
    \item The induced adjunction $(\bL\dashv \bR):\text{Ho}(\cM)\rightleftarrows\text{Ho}(\cN)$ between homotopy categories is an equivalence of categories
    \item For any cofibrant object $X\in\cM$ and any fibrant object $Y\in\cN$, a map $LX\to Y$ is a weak equivalence in $\cN$ if and only if the corresponding map $X\to RY$ under the adjunction is a weak equivalence in $\cM$
    \item Both of the following two conditions hold:
    \begin{enumerate}
        \item For every cofibrant object $X\in\cM$, the composition $X\to R(L(X))\to R(L(X)^{fib})$ (known as the derived adjunction unit) is a weak equivalence
        \item For every fibrant object $Y\in\cN$, the composition $L(R(Y)^{cof})\to L(R(Y))\to Y$ (known as the derived adjunction counit) is a weak equivalence.
    \end{enumerate}
\end{itemize}
\end{defn}

In particular, this last characterization will be important to us in demonstrating the non-equivalence between precubical sets and flows.  Another pair of characterizations (dual to one another) will be particularly useful to us going forward as well.

\begin{thm}
Consider the Quillen Pair $(L\dashv R):\cM\rightleftarrows\cN$.
\begin{itemize}
    \item If $L$ creates weak equivalences (i.e., $f\in\Mor(\cM)$ is a weak equivalence if and only if $L(f)$ is), then $(L\dashv R)$ is a Quillen equivalence if and only if for every fibrant object $Y\in\cN$, the adjunction counit $\epsilon: L(R(Y))\to Y$ is a weak equivalence.
    \item If $R$ creates weak equivalences (i.e., $f\in\Mor(\cN)$ is a weak equivalence if and only if $R(f)$ is), then $(L\dashv R)$ is a Quillen equivalence if and only if for every cofibrant object $X\in\cM$, the adjunction unit $\epsilon: X\to R(L(X))$ is a weak equivalence.
\end{itemize}
\end{thm}
\begin{proof}
A proof of the first statement (the second is essentially dual) can be found in \cite{ErIl} (Lemma 3.3).
\end{proof}

\subsection{Cofibrantly Generated Model Categories} 

A cofibrantly generated model category is a nice type of model category generated mostly by small data.  We will use these as an entry point into combinatorial model categories, which we use to define a model structure on precubical sets satisfying certain properties.

\begin{defn}
Let $\cC$ be a cocomplete category and take $S\subset\textbf{Mor}(\cC)$.  We define
\begin{itemize}
    \item $\textbf{llp}(S)$ to be the class of morphisms which has the left lifting property with respect to all morphisms in $S$.
    \item $\textbf{rlp}(S)$ to be the class of morphisms which has the right lifting property with respect to all morphisms in $S$.
    \item $\textbf{cell}(S)$ to be the class of transfinite compositions of elements of $S$.
    \item $\textbf{cof}(S):=\llp(\rlp(S))$.
\end{itemize}
\end{defn}

\begin{defn}
A model category $\cC$ is \emph{cofibrantly generated} if there are small sets of morphisms $I, J\subset\textbf{Mor}(\cC)$ such that
\begin{itemize}
    \item $I$ and $J$ admit the small object argument.
    \item $\textbf{cof}(I)$ is precisely the class of cofibrations of $\cC$
    \item $\textbf{cof}(J)$ is precisely the class of trivial cofibrations of $\cC$
\end{itemize}
\end{defn}

One very important proposition for cofibrantly generated model categories is the following.

\begin{prop}
Given a cofibrantly generated model category $\cC$, one has
\begin{itemize}
    \item $\textbf{cof}(I)$ is also the class of retracts of elements of $\cell(I)$.
    \item $\textbf{cof}(J)$ is also the class of retracts of elements of $\cell(J)$.
    \item $\rlp(I)$ is precisely the class of trivial fibrations
    \item $\rlp(J)$ is precisely the class of fibrations
\end{itemize}
\end{prop}
\begin{proof}
Found in \cite{Hir} chapter 11 (combines several propositions).
\end{proof}

Finally, before we move on to discussing combinatorial model categories, we will simply state a particularly important theorem due to Daniel Kan, which allows us to produce a cofibrantly generated model structure from the data of $I$ and $J$ given an appropriate set of weak equivalences.

\begin{thm}
Let $\cC$ be a bicomplete category and $\cW\subset\textbf{Mor}(\cC)$ closed under retracts and satisfying the 2-out-of-3 property.  If $I$ and $J$ are sets of morphisms of $\cC$ such that
\begin{itemize}
    \item Both $I$ and $J$ admit the small object argument
    \item $\cof(J)\subseteq\cof(I)\cap\cW$
    \item $\rlp(I)\subseteq\rlp(J)\cap\cW$
    \item One of $\cof(I)\cap\cW\subseteq\cof(J)$ and $\rlp(J)\cap\cW\subseteq\rlp(I)$ holds,
\end{itemize}
Then $\cC$ is a cofibrantly generated model category with weak equivalences specified by $\cW$, with $I$ a set of generating cofibrations, and $J$ a set of generating trivial cofibrations.
\end{thm}

\begin{proof}
This is \cite{Hir} Theorem 11.3.1.
\end{proof}

\subsection{Combinatorial Model categories}

In this section we will discuss a particularly nice type of model structure, generated by an extremely minimal amount of data, but with very good properties. In particular, we will see that all one needs is a class of weak equivalences and a set of generating cofibrations satisfying certain properties.

\begin{defn}
A model category $\cC$ is \emph{combinatorial} if it is a cofibrantly generated model category which is locally presentable as a category.
\end{defn}

That's it.  But this seemingly simple class of model categories admit several extremely powerful classification theorems.  We will discuss only Jeff Smith's theorem here, but there is another important classification result due to Daniel Dugger.

\begin{thm}
(Jeff Smith's Theorem)

Suppose that one has the data of
\begin{itemize}
    \item a locally presentable category $\cC$
    \item a class of morphisms $W\subset\text{Mor}(\cC)$ such that the subcategory of the arrow category of $\cC$ it defines, $\textbf{Arr}_W(\cC)\subset\textbf{Arr}(\cC)$ is an accessibly embedded accessible full subcategory
    \item a small set $I\subset\text{Mor}(\cC)$ of morphisms in $\cC$
\end{itemize}
such that
\begin{itemize}
    \item $\cW$ satisfies the 2-out-of-3 property
    \item $\text{inj}(I)\subset\cW$
    \item $\text{cof}(I)\cap\cW$ is closed under pushout and transfinite composition.
\end{itemize}
Then we have that $\cC$ is a combinatorial (and hence cofibrantly generated) model category with
\begin{itemize}
    \item weak equivalences $\cW$
    \item cofibrations $\text{cof}(I)$.
\end{itemize}
Finally, all combinatorial model structures arise in this way.
\end{thm}
\begin{proof}
Can be found in \cite{Bar} (proposition 1.7) and \cite{Bek}.
\end{proof}

This theorem essentially gives us a minimal recipe for concocting model categories, and extremely well behaved ones at that.

\subsection{left-induced model structures}

Now, the last topic we will discuss in the theory of model categories is that of induced model structures, specifically left-induced model structures (we will forego discussion of right-induced model structures, as they are irrelevant to our current topic).

\begin{defn}
Let $\cC$ be a bicomplete category and let $\cM$ be a model category.  Furthermore, suppose there is an adjunction of the form
\[
(L\dashv R):\cC\rightleftarrows \cM
\]
running between them.  The \emph{left-induced model structure} on $\cC$, if it exists, has
\begin{enumerate}
    \item weak equivalences given by those morphisms which map to weak equivalences in $\cM$ under $L$ (i.e. $L^{-1}\cW$), 
    \item cofibrations given by those morphisms which map to cofibrations in $\cM$ under $L$ (i.e. $L^{-1}\cof_\cM)$), 
    \item fibrations determined by the other two classes of morphisms.
\end{enumerate}
\end{defn}

We now cite an important theorem \cite{HKRS} which determines conditions (known as acyclicity conditions) under which a left-induced model structure exists.  We will actually state a corollary of the original theorem from \cite{HKRS}, as it is all we will need at the moment.

\begin{thm}
Suppose that $\cC$ is a bicomplete category and $\cM$ is a combinatorial model category and that there is an adjunction of the form 
\[
(L\dashv R):\cC\rightleftarrows \cM
\]
running between them.  Then the left-induced model structure on $\cC$ exists if and only if 
\[\rlp(L^{-1}\cof_\cM)\subset L^{-1}\cW_\cM.\]
\end{thm}

\begin{proof}
This is a specialization of \cite{HKRS} Proposition 2.1.4 to the situation of a combinatorial model category.
\end{proof}

\section{Flows, Precubical Sets, and Simplicial Semicategories}

The category of flows is a model for the theory of concurrency.  As we will see shortly, the category of flows may equivalently be thought of as the category of small topologically enriched semicategories.  A basic heuristic for understanding the relation between flows and concurrency is that the objects of flows correspond to possible states that a computation can be in, whereas morphism spaces represent all the possible ways of getting from one state to another.  The topology simply allows us to compare the relation between different ways of getting between states.  We want to be able to say when two execution paths between states are equivalent for the purposes of our computation, and to be able to specify precisely \emph{how} they are equivalent, thus justifying the usage not only of spaces, but more general spaces than those corresponding to mere 1-types.

\begin{defn}
A \emph{flow} $X:=(X^0,\bP X, s, t, *)$ is a quintuple consisting of a discrete set $X^0$, a locally compact space $\bP X$ called the \emph{path space}, \emph{source} and \emph{target} continuous maps $s, t:\bP X\to X^0$, and a continuous and associative \emph{path concatenation} operation
\[
*:\bP X\times_{s, t}\bP X=\{(x, y)\in\bP X^2| t(x)=s(y)\}\to\bP X.
\]
We will abuse notation and write $X$ both for the quintuple and its \emph{total space} $\bP X\sqcup X^0$.

A morphism of flows $f: X\to Y$ consists of a set map $f^0: X^0\to Y^0$ and a map of topological spaces $\bP f:\bP X\to\bP Y$ (we abuse notation and use $f$ as a stand-in for both) such that $f(s(x))=s(f(x))$, $f(t(x))=t(f(x))$, and $f(x*y)=f(x)*f(y)$ for all $x, y\in\bP X$.
\end{defn}
Together flows and maps of flows form the category $\textbf{Flow}$ of flows.

For any flow $X$, given any $\alpha, \beta\in X^0$, one may describe the path space from $\alpha$ to $\beta$ as $\bP_{\alpha, \beta}X:=\{x\in X|s(x)=\alpha\text{ and }t(x)=\beta\}$ equipped with the subspace topology (or the kaonization thereof if needed).

Furthermore, there is a functor $\text{Glob}:\textbf{Top}\to\textbf{Flow}$ which assigns to each topological space $X$ its globe $\text{Glob}(X)$, a flow such that $\text{Glob}(X)^0=\{0, 1\}$, $\bP\text{Glob}(X)=X$, $s=0$, and $t=1$.  Given some string $X_1, ..., X_n$ of topological spaces, we may "concatenate" their globes to define a new flow
\[
\text{Glob}(X_1)*\cdots*\text{Glob}(X_n)
\]
which is the flow that you get from identifying the target of the $i$th globe with the source of the $(i+1)$th globe.  In other words, if we label the flow above as $Y$, we get that $Y^0=\{0, 1, ..., n\}$ and that $\bP Y=X_1\sqcup\cdots\sqcup X_n$ such that $s|_{X_i}=i-1$ and $t|_{X_i}=i$ for $i=1\cdots n$.

\begin{thm}
$\textbf{Flow}$ is a bicomplete, topologically enriched category.
\end{thm} 
\begin{proof}
Combines Theorem 4.17 and Notation 4.14 in \cite{Gau1} (in that order).
\end{proof}

We will elide in what follows several details for the sake of brevity, and mostly refer the reader to \cite{Gau1} and \cite{Gau4} to fill in any remaining details on the fundamental theory of Flows.

\subsection{The Homotopy Theory of Flows}

\begin{defn}
Take two morphisms of flows $f, g:X\to Y$.  Then $f$ and $g$ are referred to as \emph{S-homotopic}, denoted $f\sim_S g$ if, considering $\Hom_{\textbf{Flow}}(X, Y)$ equipped with its enriched structure as a topological space, there exists a morphism 
\[h\in\Hom_{\textbf{Top}}([0, 1], \Hom_{\textbf{Flow}}(X, Y))\]
such that $h(0)=f$ and $h(1)=g$.

Two flows $X$ and $Y$ are referred to as \emph{S-homotopy equivalent} if there exists a morphism $f:X\to Y$ and a morphism $g:Y\to X$ such that $g\circ f\sim_S\id_X$ and $f\circ g\sim_S\id_Y$.
\end{defn}

There is an equivalent definition of two maps being S-homotopic that more directly parallels the first definition in topological spaces, but requires more machinery to set up.  Furthermore, our primary interest will be in weak S-homotopy equivalence.

In what follows, we take $I^{gl}_+=\{\text{Glob}(S^{n-1})\hookrightarrow\text{Glob}(D^n)\}_{n=0}^\infty\cup\{\emptyset\hookrightarrow*, *\sqcup*\to*\}$, where we follow the convention $S^{-1}=\emptyset$ and where we have $*$ be the flow consisting of a single object and the empty space of morphisms.

\begin{defn}
A map $f:X\to Y$ of flows is a \emph{weak S-homotopy equivalence} if the map $f^0:X^0\to Y^0$ of zero-skeleta is a bijection, and the map $\bP f:\bP X\to\bP Y$ is a weak-homotopy equivalence of topological spaces.  We denote the class of all weak S-homotopy equivalences by $\cW_S$.
\end{defn}

\begin{thm}
There is a combinatorial model structure on the category of flows such that $I^{gl}_+$ is the generating set of cofibrations and $\cW_S$ is the class of weak equivalences.  Furthermore, with respect to this model structure, all objects are fibrant.
\end{thm}

\begin{proof}
This is proposition 18.1 in \cite{Gau1}.
\end{proof}

\subsection{Geometric Realization of Precubical Sets and Homotopy Coherent Nerve of Flows}

Another commonly discussed model for concurrent computing is that of precubical sets. These are presheaves on the category of cubes $\square$, which may be thought of as the subcategory of $\textbf{Cat}$ given by taking all strictly non-decreasing maps between the posets $\{0\le1\}$ (including symmetry maps for the time being).

Note that there is a natural inclusion of posets into flows, where you consider the thin semicategory of the poset under strict inequality.  In other words, we have an inclusion functor $\textbf{PoSet}\hookrightarrow\textbf{Flow}$.  In particular, for any cube $\square[n]=\{0<1\}^n$, we may consider it a flow in a natural way. Now, on cubes $\square[n]$, we define their realization $|\square[n]|$ to be the cofibrant replacement of their inclusion into $\textbf{Flow}$.  Consider the category of precubical sets $\textbf{PrSh}(\square)$.  In other words, functors from the category of cubes, with morphisms only given by face operators, into sets.  There is a geometric realization functor from the category of precubical sets to flows 
\[
|-|:\textbf{PrSh}(\square)\to\textbf{Flow}\text{ such that }K\mapsto\varinjlim_{\square\downarrow K}|\square[n]|
\]

Its right adjoint is the homotopy coherent nerve of a flow, given by 
\[
N:\textbf{Flow}\to\textbf{PrSh}(\square)\text{ such that }X\mapsto(\square[n]\mapsto\Hom_{\textbf{Flow}}(|\square[n]|, X))
\]
We will prove this statement now
\begin{thm}
The geometric realization and the homotopy coherent nerve functors form an adjunction
\[
\begin{tikzcd}
\textbf{PrSh}(\square)
\arrow[r, "|-|"{name=F}, bend left=25] &
\textbf{Flow}
\arrow[l, "N"{name=G}, bend left=25]
%--- Adjunction Symbol
\arrow[phantom, from=F, to=G, "\dashv" rotate=-90]
\end{tikzcd}
\]
\end{thm}

\begin{proof}
We want to prove that for all flows $X$ and all precubical sets $K$ that we have a natural isomorphism
\[\Hom_{\textbf{Flow}}(|K|, X)\cong\Hom_{\textbf{PrSh}(\square)}(K, N(X)).\]

We do this via
\begin{align*}
    \Hom_{\textbf{Flow}}(|K|, X)&=\Hom_{\textbf{Flow}}(\varinjlim_{\square\downarrow K}|\square[n]|, X)\\
    &\cong\varprojlim_{\square\downarrow K}\Hom_{\textbf{Flow}}(|\square[n]|, X)\\
    &\cong\varprojlim_{\square\downarrow K}N(X)_n\\
    &\cong\varprojlim_{\square\downarrow K}\Hom_{\textbf{PrSh}(\square)}(\square[n], N(X))\\
    &\cong\Hom_{\textbf{PrSh}(\square)}(\varinjlim_{\square\downarrow K}\square[n], N(X))\\
    &\cong\Hom_{\textbf{PrSh}(\square)}(K, N(X))
\end{align*}
\end{proof}

Now that we have our adjunction as above, we must try to find a model structure on precubical sets that upgrades the above adjunction to a Quillen pair.

That said, there is a slightly modified notion of geometric realization that has very slightly nicer properties.  First, one should note that 

\begin{thm}
For any $n\ge0$, one has that $\bP_{0\cdots0,1\cdots1}|\partial\square[n]|$ is homotopic to $S^{n-1}$ and one has that the commutative square
\[
\begin{tikzcd}
  \text{Glob}(S^{n-1}) \arrow[r] \arrow[d, hook]
    & \mid\partial\square[n+1]\mid \arrow[d, hook] \\
  \text{Glob}(D^n) \arrow[r]
& \mid\square[n+1]\mid \end{tikzcd}
\]
is a homotopy pushout.
\end{thm}
\begin{proof}
Found in \cite{Gau2} as proposition 4.2.2.
\end{proof}

Now, this naturally makes one wonder if there is some related functor which might render this homotopy pushout into an actual pushout.  Indeed, this is the case.  We have the following theorem/definition of Gaucher.

\begin{thm}
There exists a colimit-preserving functor $gl:\textbf{PrSh}(\square)\to\textbf{Flow}$ such that for all $n$, one has a pushout square
\[
\begin{tikzcd}
  \text{Glob}(S^{n-1}) \arrow[r] \arrow[d, hook]
    & gl(\partial\square[n+1]) \arrow[d, hook] \\
  \text{Glob}(D^n) \arrow[r]
& gl(\square[n+1]) \end{tikzcd}
\].  In particular, this means the maps $gl(\partial\square[n]\hookrightarrow\square[n])$ are all cofibrations.  Furthermore, there exist natural transformations $\mu:gl\to|-|$ and $\nu:|-|\to gl$ which specialize to natural S-homotopy equivalences, and indeed are mutually naturally S-homotopy inverses, for every precubical set $K$.  Finally, there exists a weak S-homotopy equivalence of cocubical flows $gl(\square[*])\to\{0<1\}^*$.  We refer to $gl$ as the \emph{globular realization functor}.
\end{thm}
\begin{proof}
Found in \cite{Gau2} as theorem 4.2.4.
\end{proof}

In what follows, we will often use the globular realization for its particularly nice properties, although any realization will do.  What do we mean by any realization?

\begin{thm}
Suppose one has an object-wise weak equivalence of cocubical flows $X\to\{0<1\}^*$ such that $\hat X(\partial\square[n])\to\hat X(\square[n])$ is a cofibration for all $n$ (where $\hat X$ is the extension of $X$ to a functor from precubical sets to flows).  Then there exist natural transformations $\mu^X:gl\to\hat X$ and $\nu^X:\hat X\to gl$ which specialize to natural S-homotopy equivalences, and indeed are mutually naturally S-homotopy inverses, for every precubical set $K$.  Furthermore the diagram
\[
\begin{tikzcd}
  \text{Glob}(S^{n-1}) \arrow[r] \arrow[d, hook]
    & \hat X(\partial\square[n+1]) \arrow[d, hook] \\
  \text{Glob}(D^n) \arrow[r]
& \hat X(\square[n+1]) \end{tikzcd}
\]
is a homotopy pushout square.
\end{thm}
\begin{proof}
Found in \cite{Gau2} as theorem 4.2.6.
\end{proof}

Note that $gl$ has a right adjoint $N_{gl}:\textbf{Flow}\to\textbf{PrSh}(\square)$ defined and verified analogously to that of $|-|$.  Namely, we have that $N_{gl}(X)_n=\Hom_{\textbf{Flow}}(gl(\square[n]), X)$ for all $n$.  We will call this the \emph{globular nerve} of a flow.

Indeed, for any realization $\hat X$, one has a corresponding nerve $N_X$ defined analogously.

\begin{lem}
Fix a realization functor $L:\PrSh(\square)\to\Flow$.  Then for any precubical set $K$, $L(K)$ is a cofibrant flow.
\end{lem}

\begin{proof}
This proof is adapted from \cite{Gau3} proposition 7.5.  We will prove this by induction on $n$-skeleta.  Consider an arbitrary precubical set $K$.  Recall that for $L$ to be a realization as defined above, one must have that $L(\partial\square[n]\hookrightarrow\square[n])$ is a cofibration for all $n$.  In particular, we know that $L(\square[0])=*$ and that $L(\emptyset\hookrightarrow\square[0])=\emptyset\hookrightarrow*$ is a cofibration.  In particular, the fact that cofibations are closed under pushout implies that $*\to*\sqcup*$ is a cofibration, and by transfinite composition, that $\sqcup_\alpha *$ is cofibrant for any cardinal $\alpha$.  In particular, this shows that $L(K_{\le0})=L(K)^0=\sqcup_{K_0}*$ is a cofibrant flow.  Now, suppose that $L(K_{\le n-1})$ is cofibrant.  Then, as cofibrations are composed under transfinite composition, it suffices to prove that $L(K_{\le n-1})\to L(K_{\le n})$ is a cofibration.  Given that $L(\partial\square[n])\to L(\square[n])$ is a cofibration, the disjoint union of any number of copies of this morphism can be shown to be as well.  Now, note that one has the pushout diagram 

\[
\begin{tikzcd}
 \sqcup_{\alpha\in K_n}\partial\square[n]\arrow[d, hook]\arrow[r] & K_{\le n-1} \arrow[d, hook]\\
  \sqcup_{\alpha\in K_n}\square[n] \arrow[r]& K_{\le n}.
\end{tikzcd}
\]

As colimits commute with colimits, and $L$ is defined by way of colimits, we have the pushout square of flows
\[
\begin{tikzcd}
 \sqcup_{\alpha\in K_n}L(\partial\square[n])\arrow[d, hook]\arrow[r] & L(K_{\le n-1}) \arrow[d, hook]\\
  \sqcup_{\alpha\in K_n}L(\square[n]) \arrow[r]& L(K_{\le n}).
\end{tikzcd}
\]
Thus, one has that $L(K_{\le n-1})\hookrightarrow L(K_{\le n})$ is a cofibration, as cofibrations are closed under pushout.  Furthermore, as they are closed under transfinite compositions as well, we see that $L(K)$ must be cofibrant.
\end{proof}

\begin{lem}
A flow $X$ is synchronized (is bijective on constant states) if and only if it has the right-lifting property with respect to the set $\{\emptyset\hookrightarrow*, *\sqcup*\to*\}$.
\end{lem}
\begin{proof}
This is found in \cite{Gau1} (proposition 16.2), but we prove an analogous result via identical means in the following subsection.
\end{proof}

We denote a hopeful set of generating cofibrations $I=\{\partial\square[n]\hookrightarrow\square[n]\}_{n=0}^\infty\cup\{\square[0]\sqcup\square[0]\to\square[0]\}$, where $\partial\square[0]=\emptyset$.  However, we come to a deeply unfortunate fact.  While we have adjunctions between $\textbf{PrSh}(\square)$ and $\textbf{Flow}$ given by any one of the realization functors discussed above, this adjunction is not a Quillen equivalence as one might hope.

\begin{thm}
Suppose that there is a model structure on the category of precubical sets along with a realization functor $L:\textbf{PrSh}(\square)\rightleftarrows\textbf{Flow}$ which is the left Quillen adjoint in a Quillen pair $(L\dashv N_{L}):\textbf{PrSh}(\square)\rightleftarrows\textbf{Flow}$ (note that we are essentially only assuming that everything in $I$ is sent to a cofibration and that the images of $\square[n]$ are weakly equivalent to $\{0<1\}^n$ for any $n$).  Then this adjunction cannot be a Quillen equivalence.
\end{thm}
\begin{proof}
  Recall that given Quillen pair $(L\dashv R):\cC\rightleftarrows\cD$, if $(L\dashv R)$ is a Quillen equivalence, then on every fibrant object $d\in\cD$, the derived counit of the adjunction (composition of cofibrant replacement with the adjunction counit) $L(R(d)^{cof})\to L(R(d))\overset{\varepsilon_d}{\to}d$ is a weak equivalence.  Now, in our case, note that every $X\in\textbf{Flow}$ is fibrant \cite{Gau1}.  Furthermore, note that $\partial\square[n]\hookrightarrow\square[n]$ is a cofibration, and every $K\in\textbf{PrSh}(\square)$ may be built up as the transfinite composition of pushouts of such maps, starting from copies of $\emptyset\hookrightarrow\square[0]$, so every $K\in\textbf{PrSh}(\square)$ is cofibrant.  Thus, $(L\dashv N_{L}):\textbf{PrSh}(\square)\rightleftarrows\textbf{Flow}$ is a Quillen equivalence if and only if for all $X\in\textbf{Flow}$, one has that the natural map $L(N_{L}(X))\to X$ is a weak equivalence.  Now, consider $\text{Glob}(Y)$ for some $Y\in\textbf{Top}$.  Note first that the only $n$ for which one has any maps $L(\square[n])\to\text{Glob}(Y)$ are $n=0$ and $n=1$.  This is because for any $n>1$, one has that there are simply too many distinct vertices of $L(\square[n])$, and one would have to map a nonempty space to the empty space, which is impossible.  Furthermore, there are exactly two maps $L(\square[0])\to\text{Glob}(Y)$, one which singles out $0$ and one which singles out $1$.  Finally, noting that $L(\square[1])=\text{Glob}(Z)$ for some contractible space $Z$, we see that each map $L(\square[1])\to \text{Glob}(Y)$ singles out a different point of $\Hom_\textbf{Top}(Z, Y)$.  Putting all of this data together, we get that 
\[
N_{L}(\text{Glob}(Y))_n=
\begin{cases}
\{0, 1\} & n=0\\
\Hom_\textbf{Top}(Z, Y)_{disc} & n=1\\
\emptyset & n>1
\end{cases}.
\]
Now, since no higher gluing data is specified, this implies that $L(N_{L}(\text{Glob}(Y)))=\text{Glob}(\Hom_\textbf{Top}(Z, Y)_{disc})$.  Noting further that our counit in this case is a map $L(N_{L}(\text{Glob}(Y)))\to\text{Glob}(Y)$, this tells us that our counit in this case essentially amounts to a map from a discrete space to $Y$.  This can only be an equivalence in the case that $Y$ was weakly equivalent to a discrete space to begin with. Thus, taking $Y=S^2$, for example, we have a non-equivalence.
\end{proof}

\bigskip

As an aside, what we have actually demonstrated above can be used to say something stronger.  In particular, even if we replaced the weak equivalences on flows with equivalences which induce weak equivalences on path spaces and equivalences of semicategories on homotopy categories (in analogy with the $\infty$-categorical model structure on topologically enriched categories), this problem would still persist, and we could not get a Quillen equivalence between this $\infty$-semicategorical model structure and precubical sets.  Similarly, it would not work with semisimplicial sets and this new model structure either.  The unusual thing here is that on the level of cubical/simplicial sets and small topologically enriched categories, this would be precisely the Quillen equivalence between the appropriate Joyal-type model structure and the $\infty$-categorical model structure on topologically enriched categories.  Thus, the procedure of "forgetting about identities" on either side of the equivalence results in something strictly less well-behaved, and indeed, in non-Quillen equivalent model categories.  One, in fact, needs identities to make one of the core equivalences of the theory of $\infty$-categories work at all.

\bigskip

\subsection{Quillen Equivalence between Simplicial Semicategories and Flows}

\begin{rem}
In what follows, we will be dealing with several different simplicially enriched categories and at several points it will be necessary to differentiate between their sets of morphisms and their simplicial sets of morphisms.  To limit confusion, we will introduce the following piece of notation.  Suppose $\cC$ is a simplicially enriched category with $X, Y\in\cC$.  We will let $\Hom_\cC(X, Y)$ be the set of morphisms from $X$ to $Y$ and let $\cC(X, Y)$ be the corresponding simplicial set (in other words, $\Hom_\cC(X, Y):=\cC(X, Y)_0$).
\end{rem}

\begin{defn}
A \emph{simplicial semicategory} is a semicategory $K$ enriched in simplicial sets.  Simplicial semicategories and simplicial semifunctors between them form a category $\sSemiCat$.
\end{defn}

We will devote this section to proving the existence of a Quillen equivalence between $\sSemiCat$ and $\Flow$.  First, we have the following.

\begin{prop}
There exists an adjunction $(|-|\dashv Sing):\sSemiCat\rightleftarrows\Flow$.
\end{prop}
\begin{proof}
Let us begin by defining $|-|:\sSemiCat\to\Flow$.  Given any $K\in\sSemiCat$, one may define $|K|\in\Flow$ by taking the same object set, and defining for any two $x, y\in\text{ob}(K)$, $\bP_{x, y}|K|:=|\text{Map}_K(x, y)|$.  This definition also yields a corresponding definition for functors which satisfies all the appropriate associativity conditions.  Conversely, we may define $Sing:\Flow\to\sSemiCat$ as follows.  For any $X\in\Flow$, we define $Sing(X)$ by taking the same object set, and defining for any two $x, y\in X^0$, $\text{Map}_{Sing(X)}(x, y)=Sing(\bP_{x, y}X)$.  This analogously also defines a functor, and these two functors are clearly adjoints of one another.
\end{proof}

From now on, we will employ the notation $\bP^\Delta$ for our mapping simplicial sets in $\sSemiCat$ in analogy with the path spaces in $\Flow$

\begin{defn}  There exists an analogous globe functor $\text{Glob}^\Delta:\textbf{sSet}\to\sSemiCat$ to the one in the case of $\Flow$.

The \emph{generating cofibrations} of $\sSemiCat$ are 
\[
I_{Simp}:=\{\text{Glob}^\Delta(\partial\Delta^n\hookrightarrow\Delta^n)\}_{n=0}^\infty\cup\{\emptyset\hookrightarrow*, *\sqcup*\to*\}
\]
\end{defn}

Before we move on to our main propositions and theorems, we must prove the following lemma.

\begin{lem}
A map of semisimplicial categories $f:K\to L$ is synchronized (induces a bijection on object sets) if and only if it has the right lifting property with respect to $\{\emptyset\hookrightarrow*, *\sqcup*\to*\}$.
\end{lem}

\begin{proof}
Let us consider the diagrams
\[
\begin{tikzcd}
  \emptyset\arrow[d, hook, "i"]\arrow[r]& K\arrow[d, "f"]\\
  *\arrow[r, "\alpha"]\arrow[ru, dashed, "\exists\tilde{\alpha}"]& L
\end{tikzcd}
\text{ and }
\begin{tikzcd}
  *\sqcup*\arrow[d, "p"]\arrow[r, "\beta"]& K\arrow[d, "f"]\\
  *\arrow[r, "\gamma"]\arrow[ru, dashed, "\exists\tilde{\gamma}"]& L
\end{tikzcd}
\]

The first diagram is simply saying that for all $\alpha\in\Hom_{\sSemiCat}(*, L)\cong L^0$, there exists $\tilde{\alpha}\in\Hom_{\sSemiCat}(*, K)\cong K^0$ such that $f\circ\tilde{\alpha}=\alpha$.  In other words, the map $f^0:K^0\to L^0$ is surjective.

Now, let's unravel what the second diagram is saying.  Labeling the two objects in $*\sqcup*$ as $a$ and $b$ for brevity, the dashed arrow always exists if and only if for any $\beta:*\sqcup*\to K$ and $\gamma:*\to L$ such that $f\circ\beta=\gamma\circ p$ (in other words, such that $f(\beta(a))=f(\beta(b))=\gamma(*)$), there exists a $\tilde{\gamma}:*\to K$ such that
\[
\beta(a)=\tilde{\gamma}(p(a))=\tilde{\gamma}(c)=\tilde{\gamma}(p(b))=\beta(b).
\]
This, in turn holds if and only if $f^0:K^0\to L^0$ is injective.

\end{proof}

Now, we are ready to prove the main results of this section.

Before we can prove that $\sSemiCat$ is a model category equipped with a left-induced model structure, we must prove that it is bicomplete.

\begin{thm}
$\sSemiCat$ is bicomplete.
\end{thm}

\begin{proof}
This proof is analogous to the proof that $\Flow$ is bicomplete (found in \cite{Gau1}), and essentially follows from the bicompleteness of $\textbf{sSet}$.  Note first and foremost that for a set $A$ and a simplicial semicategory $K$, letting $F:\textbf{Set}\to\sSemiCat$ be the functor sending a set to the simplicial semicategory with that set as objects and no morphisms,
\[
\Hom_\textbf{Set}(A, K^0)\cong\Hom_\sSemiCat(FA, K).
\]
Thus, the object set functor is a right adjoint, and so if limits exist in $\sSemiCat$, we know exactly what their object sets must look like.  Indeed, given a (small) diagram $K_{(-)}:I\to\sSemiCat$, we can define its limit $\varprojlim_iK_i$ as follows:
\begin{itemize}
    \item $(\varprojlim_iK_i)^0=\varprojlim_i(K_i^0)$.
    \item Given any two $\alpha, \beta\in(\varprojlim_iK_i)^0$, taking $\alpha_i$ and $\beta_i$ to be their images in $K_i$ for $i\in I$, define 
    \[
    \bP^\Delta_{\alpha\beta}(\varprojlim_iK_i)=\varprojlim_i\bP^\Delta_{\alpha_i\beta_i}K_i.
    \]
    \item Given any three $\alpha,\beta,\gamma\in(\varprojlim_iK_i)^0$, define 
    \[
    *:\bP^\Delta_{\alpha\beta}(\varprojlim_iK_i)\times\bP^\Delta_{\beta\gamma}(\varprojlim_iK_i)\to\bP^\Delta_{\alpha\gamma}(\varprojlim_iK_i)
    \]
    as the limit over all $i\in I$ of the $i$th level compositions
    \[
    *_i:\bP^\Delta_{\alpha_i\beta_i}K_i\times\bP^\Delta_{\beta_i\gamma_i}K_i\to\bP^\Delta_{\alpha_i\gamma_i}K_i.
    \]
\end{itemize}
Taken altogether, this is sufficient to define limits of simplicial semicategories.  For colimits, a slightly subtler argument is needed.  In particular, given that $\sSemiCat$ is complete, we prove that colimits exist by appealing to Freyd's Adjoint Functor Theorem.  In particular, let $\Delta_I:\sSemiCat\to\sSemiCat^I$ denote the constant diagram functor from simplicial semicategories to the category of $I$-shaped diagrams in simplicial semicategories (we assume, of course, that $I$ is small).  Clearly, $\Delta_I$ commutes with limits (and is thus continuous).  We now wish to show that $\Delta_I$ satisfies the solution set condition.  Take a diagram $D\in\sSemiCat^I$.  As in Gaucher, we note that all morphisms $f: D\to\Delta_IK$ for $K\in\sSemiCat$ form a proper class of solutions, so let's try to pair this down to a set.  Now, consider the cardinal $\kappa=\aleph_0\cdot\sum_{i\in I}\#(D_i^0\sqcup\bP^\Delta D_i)$ (where here we have the cardinalities of the underlying sets).  Now, choose a representative for every isomorphism class of simplicial semicategories whose underlying object and morphism sets have cardinality less than or equal to $2^\kappa$, and let $\cA$ be the set of all these representatives.  Then $\bigcup_{A\in\cA}\Hom_{\sSemiCat^I}(D, \Delta_IA)$ forms a set of solutions, which may be proved as follows.  Consider an arbitrary natural transformation $\alpha:D\to\Delta_IK$ for some $K\in\sSemiCat$.  Then for every $i\in I$, we get a map $\alpha_i:D_i\to K$.  Thus, we obtain a sub-simplicical semicategory $\langle\bigcup_{i\in I}\alpha_i(D_i)\rangle\subset K$ generated by the images of the $D_i$ with overall cardinality less than or equal to $2^\kappa$.  Thus, $\langle\bigcup_{i\in I}\alpha_i(D_i)\rangle\cong A$ for some $A\in\cA$, and our morphism must factor through some map $\beta:D\to\Delta_IA$ for $A\in\cA$.
\end{proof}

Now we can continue on with left-induction of the model structure.

\begin{prop}
The model structure on $\Flow$ may be left-induced via the adjunction introduced in the previous proposition.  This upgrades $(|-|\dashv Sing):\sSemiCat\rightleftarrows\Flow$ into a Quillen pair.
\end{prop}

\begin{proof}
Define $\cW_{\sSemiCat}$ to consist of all those morphisms that become weak equivalences under realization.  Recall that $\Flow$ is a combinatorial model category, and hence the model structure on $\Flow$ may be left-induced along $(|-|\dashv Sing):\sSemiCat\rightleftarrows\Flow$ if and only if $|\rlp(|-|^{-1}(\cof_\Flow))|\subset\cW_\Flow$.  Note that by definition, for all $n$, $|\text{Glob}^\Delta(\partial\Delta^n\hookrightarrow\Delta^n)|\cong\text{Glob}(\partial D^n\hookrightarrow D^n)$.  Now, this implies that $I_{Simp}\subset|-|^{-1}(\cof_{\Flow})$, which in turn yields that $\rlp(I_{Simp})\supset\rlp(|-|^{-1}(\cof_{\Flow}))$.  Thus, if we can show that $|rlp(I_{Simp})|\subset\cW_{\Flow}$, then we are done.  Suppose that we have a map $(f:K\to L)\in\rlp(I_{Simp})$.  Then, first of all, because $f$ has the right lifting property with respect to $\{\emptyset\hookrightarrow*, *\sqcup*\to*\}$, it is a synchronized morphism, or in other words, $\text{ob}(f):\text{ob}(K)\to\text{ob}(L)$ is a bijection of sets.  Furthermore, for all $n$, and for all commutative squares 
\[
\begin{tikzcd}
  \text{Glob}^\Delta(\partial\Delta^n) \arrow[d, hook]\arrow[r]& K\arrow[d, "f"]\\
  \text{Glob}^\Delta(\Delta^n)\arrow[r]\arrow[ru, dashed]& L
\end{tikzcd},
\] the dashed arrow exists.  Note that this last condition holds if and only if $\bP^\Delta f:\bP^\Delta K\to\bP^\Delta L$ is a trivial Kan fibration of simplicial sets.  However, this implies that $|\bP^\Delta f|=\bP|f|$ is a trivial Serre fibration of topological spaces.  This, plus the fact that $|f|^0: |K|^0\to|L|^0$ is a bijection of sets, yields that $|f|$ is a trivial fibration of flows, and hence that $|f|\in\cW_\Flow$.  Thus, we have shown that $|\rlp(I_{Simp})|\subset\cW_\Flow$, and hence that the left-induced model structure on $\sSemiCat$ exists.
\end{proof}

Furthermore this Quillen pair is actually a Quillen equivalence.

\begin{thm}
The Quillen adjunction $(|-|\dashv Sing):\sSemiCat\rightleftarrows\Flow$ is a Quillen equivalence.
\end{thm}

\begin{proof}

To start with, note that since our model structure on $\sSemiCat$ is left induced, the left-adjoint realization functor creates weak equivalences.  This implies that we only need to show that for every fibrant $X\in\Flow$ (in other words, any flow), the counit of the adjunction $\epsilon: |Sing(X)|\to X$ is a weak equivalence.  However, this holds due to the same result for the Quillen adjunction between simplicial sets and topological spaces.
\end{proof}

Thus, we actually have a nice combinatorial model for Flows, much akin to that for spaces.  Furthermore, we have the following results.

\begin{thm}
There is a combinatorial model structure on $\sSemiCat$ which has as its generating set of cofibrations $I_{Simp}$ Quillen equivalent to the structure constructed above via the identity.
\end{thm}

\begin{proof}
First, we briefly show using Jeff Smith's theorem that we have a valid combinatorial model category structure.  Note that $\rlp(I_{Simp})\subset\cW_{\sSemiCat}$, that $\cW_{\sSemiCat}$ satisfies 2-out-of-3, and that $\cW_{\sSemiCat}$ and $\cof(I_{Simp})$ are both closed under pushout and transfinite composition.  Thus, there is a valid combinatorial model category structure on $\sSemiCat$ with $I_{Simp}$ as its generating cofibrations, and $\cW_{\sSemiCat}$ as its weak equivalences.

We now prove that it is Quillen equivalent to the model category structure from above.  Note first that by definition the identity functor is self-adjoint.  Now, observe that since $I_{Simp}\subset|-|^{-1}(\cof_{\Flow})$, one has that $\rlp(|-|^{-1}(\cof_{\Flow}))\subset\rlp(I_{Simp})$ and furthermore that $\cof(I_{Simp})\subset|-|^{-1}(\cof_{\Flow})$.  Thus, by definition, taking the identity as a left adjoint, it takes cofibrations into cofibrations, and taking the identity as a right adjoint, it takes fibrations into fibrations.  Thus, the identity functor on $\sSemiCat$ is a Quillen self-adjunction between the two model structures we have discussed thus far.  Since these two model structures have precisely the same weak equivalences, it is automatic that the adjunction unit between cofibrant objects is a weak equivalence.  Indeed, it is the identity.  Thus, the identity functor forms a Quillen equivalence with itself between these two model structures on $\sSemiCat$.
\end{proof}

\begin{rem}
In fact, this same technique yields a proof for demonstrating the equivalence of simplicial categories and topological categories found in \cite{Ili}.  The one difference is in showing that $|\rlp(|-|^{-1}\cof_{\Cat(\textbf{Top})})|\subset\cW_{\Cat(\textbf{Top})}$, but this is a relatively simple alteration.
\end{rem}

\subsection{A Few Properties of Simplicial Semicategories}

As was hinted at in the previous section, simplicial semicategories admit a number of definitions analogous to those found in $\Flow$, for example globes, path simplicial sets, and so on.  We will simply use the same notation in what follows.

Note that $\sSemiCat$ obeys the same formal properties as $\Flow$.  That said, there are many situations in which simplicial semicategories are particularly well-behaved.  For example, they have an extremely natural notion of simplicially enriched Hom-set, and many of the same theorems admit a shorter proof, with several conditions being removed as opposed to their counterparts in flows.  In particular, see the proofs of proposition 3.25 and theorems 3.26 through 3.28 below.

\begin{defn}
Given a simplicial set $S$ and $K\in\sSemiCat$, we define $\{S, K\}\in\sSemiCat$ as follows:
\begin{itemize}
    \item $\{S, K\}^0:=K^0$
    \item For any two $\alpha,\beta\in K^0$, $\bP^\Delta_{\alpha\beta}\{S, K\}:=\textbf{sSet}(S, \bP^\Delta_{\alpha\beta}K)$.
    \item for any three $\alpha, \beta, \gamma\in K^0$, the composition law is the composite 
    \[
    *:\bP^\Delta_{\alpha\beta}\{S, K\}\times\bP^\Delta_{\beta\gamma}\{S, K\}\cong\textbf{sSet}(S,\bP^\Delta_{\alpha\beta}K\times\bP^\Delta_{\beta\gamma} K)\to\textbf{sSet}(S,\bP^\Delta_{\alpha\gamma} K)=\bP^\Delta_{\alpha\gamma}\{S, K\}
    \]
\end{itemize}
\end{defn}

\begin{thm}
The assignment $\{-, -\}:\textbf{sSet}\times\sSemiCat\to\sSemiCat$ is contravariantly functorial in the first argument and covariantly functorial in the second argument.  Furthermore, one has the natural isomorphisms $\{S, \varprojlim_iK_i\}\cong\varprojlim_i\{S, K_i\}$ and $\{\varinjlim_iS_i, K\}\cong\varprojlim_i\{S_i, K\}$.  Finally, given any two $S, T\in\textbf{sSet}$, one has for all $K\in\sSemiCat$ that $\{S\times T, K\}\cong\{S, \{T, K\}\}$.
\end{thm}

\begin{proof}
The functoriality is clear, the behavior with respect to limits in both arguments follows from the behavior of internal homs in $\textbf{sSet}$ with respect to limits in both arguments, and the last condition follows from the adjunction in $\textbf{sSet}$ between internal hom and cartesian product.
\end{proof}

This pairing actually yields the following theorem:

\begin{prop}
$\sSemiCat$ is simplicially enriched, and the assignment $\sSemiCat^{op}\times\sSemiCat\to\textbf{sSet}$ given by $(K, L)\mapsto\sSemiCat(K, L)$ is functorial.
\end{prop}

\begin{proof}
We adapt Joyal's discussion of the enrichement of simplicial categories found in \cite{Joy} to the setting of simplicial semicategories.  We first show that for any $n\in\N$, the functor $\{\Delta^n, -\}:\sSemiCat\to\sSemiCat$ described above defines a monad.  To show this, we will employ the contravariant functoriality of the first argument of $\{-,-\}$.  Consider the evident morphism $\Delta^n\to\Delta^0$.  For any $K\in\sSemiCat$, this provides us with a unit map $K\to\{\Delta^n, K\}$ upon noting the natural isomorphism $\{\Delta^0, K\}\cong K$.  Our multiplication map arises from the diagonal $\Delta^n\hookrightarrow\Delta^n\times\Delta^n$ via the composition
\[
\{\Delta^n, \{\Delta^n, K\}\}\to\{\Delta^n\times\Delta^n, K\}\to\{\Delta^n, K\}.
\]
Now for any $K, L\in\sSemiCat$, let us define $\sSemiCat(K, L)$ via the assignment $\sSemiCat(K, L)_n=\Hom_\sSemiCat(K, \{\Delta^n, L\})$.  We can define composition for $K, L, M\in\sSemiCat$ as a Kleisli multiplication

\[
\sSemiCat(K, L)_n\times\sSemiCat(L, M)_n\to\sSemiCat(K, M)_n.
\]
Explicitly, this is the following composition, where we omit Hom subscripts for brevity:
\[
\begin{tikzcd}[column sep=0]
  \sSemiCat(K, L)_n\times\sSemiCat(L, M)_n \arrow[equal]{r}
\arrow[d, phantom, ""{coordinate, name=Z}]
& \Hom(K, \{\Delta^n, L\})\times\Hom(L, \{\Delta^n, M\}) 
\arrow[dl,
rounded corners,
to path={ -- ([xshift=2ex]\tikztostart.east)
|- (Z) [near end]\tikztonodes
-| ([xshift=-2ex]\tikztotarget.west) -- (\tikztotarget)}] \\
  \Hom(K, \{\Delta^n, L\})\times\Hom(\{\Delta^n, L\}, \{\Delta^n,\{\Delta^n, M\}\}) \arrow[r]
\arrow[d, phantom, ""{coordinate, name=W}]
& \Hom(K, \{\Delta^n,\{\Delta^n, M\}\}) \arrow[dl,
rounded corners,
to path={ -- ([xshift=2ex]\tikztostart.east)
|- (W) [near end]\tikztonodes
-| ([xshift=-2ex]\tikztotarget.west) -- (\tikztotarget)}] \\
\Hom(K, \{\Delta^n, M\})\arrow[equal]{r}&\sSemiCat(K, M)_n.
&\,
\end{tikzcd}
\]
Taken together, this determines the simplicial set $\sSemiCat(K, L)$, and the composition law $\sSemiCat(K, L)\times\sSemiCat(L, M)\to\sSemiCat(K, M)$.  Given this definition, functoriality is immediate.
\end{proof}

Moreover, one has the following theorem.

\begin{thm}
The functor $\{S, -\}:\sSemiCat\to\sSemiCat$ has a left adjoint denoted by $S\boxtimes(-)$.  This defines a bifunctor $(-)\boxtimes(-)$ Furthermore, this has the following properties
\begin{itemize}
    \item There is a natural isomorphism of simplicial semicategories given by 
    \[
    S\boxtimes(\varinjlim_iK_i)\cong\varinjlim(S\boxtimes K_i).
    \]
    \item $\Delta^0\boxtimes K\cong K$.
    \item $S\boxtimes$Glob$(T)\cong$Glob$(S\times T)$.
    \item There is a natural bijection $(S\boxtimes K)^0\cong K^0$.
    \item $(S\times T)\boxtimes K\cong S\boxtimes(T\boxtimes K)$.
\end{itemize}
\end{thm}

\begin{proof}
We begin by proving the existence of a left adjoint at all.  First, note that for any simplicial set $S$, one has that $\{S, -\}$ commutes with small limits.  Thus, we need only verify the solution set condition.  Begin by choosing a simplicial semicategory $K$.  As in the proof that $\sSemiCat$ is cocomplete, we start by analyzing the class of solutions $f:K\to\{S, L\}$ for all $L\in\sSemiCat$ and all $f\in\Hom_\sSemiCat(K, \{S, L\})$.  Now, consider the cardinal $\kappa=\#S\cdot(\#K_0+\#\bP^\Delta K)$.  By definition, if $K$ is nonempty,  $\kappa\ge\aleph_0$, because the underlying set of a simplicial set is nonempty in a countably infinite number of degrees.  Now, choose a representative of every isomorphism class of simplicial semicategories whose object and morphism simplicial sets have underlying set cardinality less than or equal to $2^\kappa$, and denote the set of all these representatives by $\cA$.  Now, we verify that $\bigcup_{A\in\cA}\Hom_\sSemiCat(K, \{S, A\})$ form a set of solutions.  Consider an arbitrary $f:K\to\{S, M\}$ for some simplicial semicategory $M$.  Now, we let $N\subset M$ be the subsimplicial semicategory generated by elements of the form $f(K)(S)$.  We know by definition that $\#N\le2^\kappa$.  Thus, in particular, $\#\{S, N\}\le\#S\cdot2^\kappa=2^\kappa$ (given that $\kappa\ge\aleph_0$), and hence, $\{S, N\}\cong\{S,A\}$ for some $A\in\cA$.  Thus, our initial morphism factors through our solution set, and we have a well-defined left adjoint, which we denote by $S\boxtimes(-)$.
\begin{itemize}
    \item Note that for all $L\in\sSemiCat$, one has that 
    \begin{align*}
    \Hom_\sSemiCat(S\boxtimes(\varinjlim_iK_i),L)&\cong\Hom_\sSemiCat(\varinjlim_iK_i,\{S, L\})\\
    &\cong\varprojlim_i\Hom_\sSemiCat(K_i,\{S, L\})\\
    &\cong\varprojlim_i\Hom_\sSemiCat(S\boxtimes K_i,L)\\
    &\cong\Hom_\sSemiCat(\varinjlim_i(S\boxtimes K_i),L),
    \end{align*}
    which implies a natural isomorphism.
    \item Similarly to the above, one has
    \[
    \Hom_\sSemiCat(\Delta^0\boxtimes K, L)\cong\Hom_\sSemiCat(K, \{\Delta^0, L\})\cong\Hom_\sSemiCat(K, L).
    \]
    \item This arises from the adjunction $\Hom_{\textbf{sSet}}(K\times L, M)\cong\Hom_{\textbf{sSet}}(K, \textbf{sSet}(L, M))$ on the level of simplicial sets.
    \item  Denoting by $*$ the simplicial semicategory with one object and no morphisms, note that by definition, $K^0\cong\Hom_\sSemiCat(*, K)$, and considered as a simplicial semicategory, $K^0\cong\sqcup_{*\to K}*$.  Thus, as $S\boxtimes(-)$ commutes with colimits, we only need to show that $S\boxtimes*\cong*$.  This follows from
    \[
    \Hom_\sSemiCat(S\boxtimes*,L)\cong\Hom_\sSemiCat(*,\{S, L\})\cong\Hom_\sSemiCat(*,L),
    \]
    since one always has $\{S, L\}^0\cong L^0$ by construction.
    \item Finally, one has that
    \begin{align*}
    \Hom_\sSemiCat(S\boxtimes(T\boxtimes K),L)&\cong\Hom_\sSemiCat(T\boxtimes K,\{S, L\})\\
    &\cong\Hom_\sSemiCat(K,\{T, \{S, L\}\})\\
    &\cong\Hom_\sSemiCat(K,\{T\times S, L\})\\
    &\cong\Hom_\sSemiCat((S\times T)\boxtimes K, L).
    \end{align*}
\end{itemize}

\end{proof}

\begin{thm}
For all $S\in\textbf{sSet}$ and all $K, L\in\sSemiCat$, one has that
\[
\Hom_\textbf{sSet}(S, \sSemiCat(K, L))\cong\Hom_\sSemiCat(K, \{S, L\})\cong\Hom_\sSemiCat(S\boxtimes K, L).
\]
\end{thm}

\begin{proof}
Given that the second equivalence in the theorem follows from the mere fact of having an adjunction, we can simply focus on the first equivalence.  Note that one may write any simplicial set $S$ naturally as a colimit $S\cong\varinjlim_{\Delta\downarrow S}\Delta^n$.  This, in turn, yields the following:
\begin{align*}
    \Hom_\textbf{sSet}(S, \sSemiCat(K, L))&\cong\Hom_\textbf{sSet}(\varinjlim_{\Delta\downarrow S}\Delta^n, \sSemiCat(K, L))\\
    &\cong\varprojlim_{\Delta\downarrow S}\Hom_\textbf{sSet}(\Delta^n, \sSemiCat(K, L))\\
    &\cong\varprojlim_{\Delta\downarrow S}\sSemiCat(K, L)_n\\
    &=\varprojlim_{\Delta\downarrow S}\Hom_\sSemiCat(K, \{\Delta^n, L\})\\
    &\cong\Hom_\sSemiCat(K, \varprojlim_{\Delta\downarrow S}\{\Delta^n, L\})\\
    &\cong\Hom_\sSemiCat(K, \{\varinjlim_{\Delta\downarrow S}\Delta^n, L\})\\
    &\cong\Hom_\sSemiCat(K, \{S, L\}).
\end{align*}
\end{proof}

\begin{thm}
The enriched Hom simplicial sets in $\sSemiCat$ behave "as one would expect" with respect to limits and colimits.  Namely:
\begin{itemize}
    \item $\sSemiCat(\varinjlim_iK_i, L)\cong\varprojlim_i\sSemiCat(K_i, L)$ for any colimit.
    \item $\sSemiCat(K, \varprojlim_iL_i)\cong\varprojlim_i\sSemiCat(K, L_i)$ for any limit.
\end{itemize}
\end{thm}

\begin{proof}
We will only prove the first statement, as the second may be proven analogously.  Note that by Yoneda, $\sSemiCat(K, L)_n\cong\Hom_\textbf{sSet}(\Delta^n,\sSemiCat(K,L))$ and by the above, $\Hom_\textbf{sSet}(\Delta^n,\sSemiCat(K,L))\cong\Hom_\sSemiCat(\Delta^n\boxtimes K,L)$,  Now, note that 
\begin{align*}
    \sSemiCat(\varinjlim_iK_i, L)_n&\cong\Hom_\textbf{sSet}(\Delta^n,\sSemiCat(\varinjlim_iK_i,L))\\
    &\cong\Hom_\sSemiCat(\Delta^n\boxtimes(\varinjlim_iK_i),L)\\
    &\cong\Hom_\sSemiCat(\varinjlim_i(\Delta^n\boxtimes K_i),L)\\
    &\cong\varprojlim_i\Hom_\sSemiCat(\Delta^n\boxtimes K_i,L)\\
    &\cong\varprojlim_i\Hom_\textbf{sSet}(\Delta^n,\sSemiCat(K_i,L))\\
    &\cong\Hom_\textbf{sSet}(\Delta^n,\varprojlim_i\sSemiCat(K_i,L))\\
    &\cong(\varprojlim_i\sSemiCat(K_i, L))_n,
\end{align*}
where equivalence of the unenriched homs in this string results from abstract nonsense.  Now, finally, note that for any $S\in\textbf{sSet}$, one has that $S\cong\varinjlim_{\Delta\downarrow S}\Delta^n$.  Thus, making use of the above equivalences, we have
\begin{align*}
    \Hom_\textbf{sSet}(S,\sSemiCat(\varinjlim_iK_i,L))&\cong\Hom_\textbf{sSet}(\varinjlim_{\Delta\downarrow S}\Delta^n,\sSemiCat(\varinjlim_iK_i,L))\\
    &\cong\varprojlim_{\Delta\downarrow S}\Hom_\textbf{sSet}(\Delta^n,\sSemiCat(\varinjlim_iK_i,L))\\
    &\cong\varprojlim_{\Delta\downarrow S}\Hom_\textbf{sSet}(\Delta^n,\varprojlim_i\sSemiCat(K_i,L))\\
    &\cong\Hom_\textbf{sSet}(\varinjlim_{\Delta\downarrow S}\Delta^n,\varprojlim_i\sSemiCat(K_i,L))\\
    &\cong\Hom_\textbf{sSet}(S,\varprojlim_i\sSemiCat(K_i,L)).
\end{align*}
Thus, we have that there is a natural isomorphism between the simplicial sets $\sSemiCat(\varinjlim_iK_i,L)$ and $\varprojlim_i\sSemiCat(K_i,L)$.
\end{proof}

We now end our discussion with an important definition and a theorem.

\begin{defn}
Two simplicial semifunctors $f, g: K\to L$ are simplicial S-homotopy equivalent if there exists $H\in\Hom_\sSemiCat(\Delta^1\boxtimes K, L)$ such that $H|_0=f$ and $H|_1=g$.  Equivalently, if there exists $h\in\Hom_\textbf{sSet}(\Delta^1, \sSemiCat(K, L))$ such that $h(0)=f$ and $h(1)=g$.
\end{defn}

Finally, we note the following theorem.

\begin{thm}
The functor $\Delta^1\boxtimes(-)$ is a cylinder functor when equipped with the natural transformations $e_0:\{0\}\boxtimes(-)\to\Delta^1\boxtimes(-)$ and $e_1:\{1\}\boxtimes(-)\to\Delta^1\boxtimes(-)$ and the natural transformation $p:\Delta^1\boxtimes(-)\to\{0\}\boxtimes(-)$ induced by projection.
\end{thm}

\begin{proof}
This is relatively clear.  One merely notes for $i\in\{0, 1\}$ that $p\circ e_i$ is naturally equivalent to the identity transformation.
\end{proof}

\section{Tree-like Flows and Boxed tree sets}

Now, analyzing the cubical nerve above, it has a natural extension to the study of branching concurrent processes.  In the following section, we will introduce two homotopy-coherent operations between the category of flows, and the category of (pre-)boxed tree sets and a cubical/boxed tree analogue.  We denote the category of finite symmetric trees considered as posets by $\cT$.  Finite trees with only injective morphisms between them will be denoted by $\cT_{inj}$.  Now, let us note that since we have the obvious inclusion of the category of posets into flows, we have an inclusion of $\cT_{inj}$ into $\textbf{Flow}$.  Now, taking the cofibrant replacement of everything, we obtain a “geometric realization” of every finite tree $T$, which we denote $|T|_{\cT}$.  Using these realizations, we can cook up homotopy coherent nerve objects between flows and pre-tree-sets.  

\begin{rem}
These tree sets are \emph{not} dendroidal sets, as this category $\cT$ is not the category $\Omega$ of Moerdijk and Weiss.
\end{rem}

Our realization is given as follows.  For all $K\in\textbf{PrSh}(\cT_{inj})$, we obtain a flow given by $|K|_{\cT}:=\varinjlim_{\cT_{inj}\downarrow K}|T|_{\cT}$.
Similarly, our nerve is given for all flows by $N(X)_T:=\Hom_{\textbf{Flow}}(|T|_{\cT}, X)$.  These functors form an adjunction
\[
\begin{tikzcd}
\textbf{PrSh}(\cT_{inj})
\arrow[r, "|-|_{\cT}"{name=F}, bend left=25] &
\textbf{Flow}
\arrow[l, "N_{\cT}"{name=G}, bend left=25]
%--- Adjunction Symbol
\arrow[phantom, from=F, to=G, "\dashv" rotate=-90]
\end{tikzcd}.
\]
This is proved much in the same way as the adjunction before is.

Now, we can define a category which will aid us in our attempts to understand the interactions between concurrent and sequential processes.

\subsection{The Category of Boxed Symmetric Trees}

\begin{defn}
The category of \emph{boxed symmetric trees} $\boxed{\cT}$ is the category whose objects consist of $n$-tuples of elements of $\cT$ for varying $n\in\N$.  Morphisms are generated by the following types of arrow:

i.  For any $n$-tuple $T_1\times\cdots\times T_n\in\boxed{\cT}$, and for any $\{f_i:T_i\to T'_i\}\in T_i\downarrow\cT$ for any $i\in\{0, ..., n\}$, the map
\[
f_i:T_1\times\cdots\times T_i\times\cdots\times T_n\to T_0\times\cdots\times T'_i\times\cdots\times T_n
\]
given by applying the map $\sigma_i$ to the $i$th coordinate and leaving the others unchanged.

ii.  For any $n$-tuple $T_1\times\cdots\times T_n\in\boxed{\cT}$, any $\varepsilon\in\{0, 1\}$, and any $i\in\{1, ..., n+1\}$, one has 
\[
\partial_i^\varepsilon:T_1\times\cdots\times T_i\times\cdots\times T_n\to T_1\times\cdots\times T_{i-1}\times[1]\times T_i\times\cdots\times T_n
\]
via $\partial_i^\varepsilon(a_1, ..., a_n)=(a_1, ..., a_{i-1},\varepsilon, a_i, ..., a_n)$.

iii.  For any $n$-tuple $T_1\times\cdots\times T_n\in\boxed{\cT}$ and any $i\in\{0, ..., n\}$, one has 
\[
s_i:T_1\times\cdots\times T_n\to T_1\times\cdots\times \hat{T_i}\times\cdots\times T_n
\]
given by omitting the $i$th coordinate.

iv. For all $\sigma\in\Sigma_n$ we obtain the obvious map
\[
\sigma:T_1\times\cdots\times T_n\to T_{\sigma(1)}\times\cdots\times T_{\sigma(n)}
\]
permuting the different factors.

We define $\boxed{\cT}_{inj}$ to be the subcategory of $\boxed{\cT}$ defined by taking as generators only the injective morphisms described above (i.e. only the injective tree morphisms in i and morphisms in iii).
\end{defn}

In more informal language, $\boxed{\cT}$ consists of cubes where we allow as the sides not just the standard interval, but in fact all finite trees as our intervals.  Furthermore, we prune our trees and grow branches.

When describing what simplicial nerves describe in the categorical or homotopy coherent categorical setting, we see that they correspond to chaining together composable arrows, in the setting above, they would correspond to running computations, not concurrently, but sequentially.

Cubical nerves, on the other hand, model processes running concurrently, where each independent direction corresponds to a different operation being run at the same time.

What we hope to achieve with $\boxed{\cT}$ is to describe chained concurrent processes, possibly each with their own "flowchart" allowing for branched procedures in each of the factors.  In what follows, we briefly ponder the categories of $\boxed{\cT}$-sets and pre $\boxed{\cT}$-sets, before trying to understand their geometric realization into flows and the homotopy coherent nerve back.

We define the presheaf categories $\textbf{PrSh}(\boxed{\cT})$ and $\textbf{PrSh}(\boxed{\cT}_{inj})$ to be the categories of $\boxed{\cT}$-sets and pre $\boxed{\cT}$-sets respectively.

We may also consider a slightly larger category of shapes, which we can call $\mathfrak T$.  We may define this as the full subcategory of $\mathbf{FinPoSet}$ generated by the objects $\Pi_iT_i$ as in $\boxed{\cT}$.  This, in particular has "connection-like" morphisms built into it, among other things.  It allows for a slightly wider set of computational interpretations than $\boxed{\cT}$, as illustrated by the following idea.  Given any $T_1\times T_2\in\mathfrak T$ consisting of the product of two trees, one has a morphism $T_1\times T_2\to\{0, 1\}$ in $\mathfrak T$ given by $(s,t)\mapsto 0$ if $s$ and $t$ are both the root, and $(s, t)\mapsto 1$ if else, which corresponds roughly to checking if both computations involved have initialized or not.

\subsection{The category of Boxed Trees is a Test Category}

We briefly recall the notion of a test category before demonstrating that the category of boxed trees is a test category.  Test categories were first introduced by Grothendieck in \cite{Gro} in order to come up with reasonable combinatorial models for spaces (a particularly nice introduction can be found in\cite{Cis}).  In particular, a test category can be thought of as a small category with the property that all homotopy types may be modeled by presheaves on it.  This is done in the following manner.

Recall that $\Cat$ is the category of small categories.  Let us define $\cW_\infty$ to be the class of "weak equivalences of categories."  Namely, these are functors which become weak homotopy equivalences under the nerve functor into simplicial sets (in other words, the $\infty$-groupoidifications of these categories are equivalent).  Note that while certainly equivalences of categories are weak equivalences in this manner, it is a much wider class of functors, including any functor which is a left or right adjoint, among others (this is shown by noting that natural transformations are mapped via the nerve construction to simplicial homotopies, which ensures that the unit and counit map to a homotopy equivalence).  An important theorem is that the localization of $\Cat$ by $\cW_\infty$ is equivalent to the standard homotopy category of CW complexes/simplicial sets (in fact, there is a model structure on $\Cat$ due to Thomason \cite{Tho} which realizes this equivalence as a Quillen equivalence).

Consider a small category $\cC$.  Note that there is a natural adjunction
\[
(|-|_\cC\dashv N_\cC):\PrSh(\cC)\rightleftarrows\Cat
\]
defined in one direction by taking for every $C\in\PrSh(\cC)$, $|C|_\cC=\cC\downarrow C$, and in the other direction by taking $N_\cC(\cD)_c=\Hom_\Cat(\cC\downarrow c,\cD)$.

\begin{defn}
We may define \emph{weak test categories} as those small categories $\cC$ for which the counit of the adjunction above $|N_\cC(\cD)|_\cC\to\cD$ is always a weak equivalence.
\end{defn}

Now, we may further analyze the adjunction $(h\dashv N):\textbf{sSet}\rightleftarrows\Cat$.  Note that we have $N(|-|_\cC)\dashv N_\cC\circ h$.

\begin{defn}
If this composite adjunction may be upgraded to a Quillen equivalence with a model structure on $\PrSh(\cC)$ whose cofibrations are monomorphisms, we then say that $\cC$ is a \emph{test category} (This was noted to be equivalent to the more technical definition given below in \cite{ArCiMo}).  In other words, test categories are precisely those which provide a good combinatorial model of spaces upon taking presheaves.
\end{defn}

There are numerous equivalent classifications of (weak) test categories which provide concrete criteria which may be checked (sacrificing brevity and ease of understanding for an actual ability to perform calculations).  The following can be found in \cite{Cis}.

\begin{prop}
A category $\cC$ is a \emph{test category} if and only if the following conditions hold:
\begin{enumerate}
    \item $\cC$ is aspherical (i.e. $N(\cC)$ is a contractible simplicial set.
    \item One of these equivalent conditions hold
    \begin{enumerate}
        \item $\cC$ is a local test category (for every object $c\in\cC$, the overcategory $\cC/c$ is a weak test category, which in turn means that.
        \item The subobject classifier $L_\cC$ in $\PrSh(\cC)$ is locally aspherical.
        \item there exists a locally aspherical separating interval in $\textbf{PrSh}(\cC)$.
    \end{enumerate}
\end{enumerate}
\end{prop}

We will now provide some of the necessary definitions from the proposition above.

\begin{defn}
Given $\PrSh(\cC)$ as above, an \emph{interval} in $\cC$ is a triple $(I, d_0, d_1)$, where $I\in\PrSh(\cC)$ and $d_i:*_\cC\to I$ for $i=0,1$, where $*_\cC$ is the terminal object of $\PrSh(\cC)$.  This interval is called \emph{aspherical} if $|I|_\cC$ is weakly equivalent to the terminal category (in other words, $I$ is aspherical as a homotopy type), and is called separating if the equalizer of the double arrow $(d_0, d_1)$ is the empty presheaf on $\cC$.
\end{defn}

To prove that $\boxed{\cT}$ is a test category, it suffices to prove that $\boxed{\cT}$ is both aspherical and a local test category.  Let's work this out more concretely.

\begin{lem}
Any small category $\cC$ with a terminal object is acyclic.
\end{lem}

\begin{proof}
Let $\cC$ be a small category which has a terminal object $e\in\cC$.  There exists a natural transformation from the identity functor to the constant functor at $e$ whose components are the unique maps to $e$.  Furthermore, upon taking the nerve of $\cC$, this natural transformation becomes a homotopy from $N(\cC)$ to a point, yielding the result.
\end{proof}

Now, due to the above, we obtain that $\boxed{\cT}$ is acyclic, since $[0]$ is a terminal object for $\boxed{\cT}$.

\begin{defn}
Given a small category $\cC$, a \emph{functorial precylinder} is a triple $(I, \partial^0, \partial^1)$ such that $I:\cC\to\cC$ is an endofunctor of $\cC$ and $\partial^i$ are natural transformations from the identity functor on $\cC$ to $I$.

An \emph{augmentation} of $(I, \partial^0, \partial^1)$ is a collection of morphisms $\sigma_a: I(a)\to a$ for all objects $a\in\cC$ such that for $i=0, 1$ and $a\in\cC$, one has that $\sigma_a\circ\partial^i_a=\id_a$.  A precylinder which may be equipped with an augmentation is called \emph{augmented}.
\end{defn}

\begin{prop}
Fix a small category $\cC$, a functor $i:\cC\to\Cat$, and a functorial augmented precylinder $(I, \partial^0, \partial^1)$. Suppose the following conditions are satisfied:
\begin{itemize}
    \item For all $a\in \cC$, the functor $i\circ\partial_a^0:i(a)\to iI(a)$ is an open immersion (i.e. there exists an isomorphism between $i(a)$ and a sieve $U_a$ of the category $iI(a)$), and the functor $i\circ\partial_a^1:i(a)\to iI(a)$ factorizes through the complementary cosieve of $U_a$, which we denote by $F_a=iI(a)-U_a$.
    \item For all morphisms $\alpha: a\to a'$ in $\cC$, one has $iI(\alpha)(F_a)\subset F_{a'}$.
    \item For all $a\in \cC$, the category $i(a)$ has a final object.
\end{itemize}
then $i$ is a local test functor and $\cC$ is therefore a local test category.
\end{prop}
\begin{proof}
This is proven in \cite{Cis} lemma 8.4.12.
\end{proof}

In particular, this implies the following lemma.
\begin{lem}
The category $\boxed{\cT}$ is a local test category.
\end{lem}
\begin{proof}
We take for simplicity all elements of $\cT$ to be directed towards the root.  There is a natural embedding of $\boxed{\cT}$ into $\Cat$, which we will label $i:\boxed{\cT}\hookrightarrow\Cat$, in which we take the elements of $\boxed{\cT}$ and map them to their corresponding poset categories.  Note that for every $\Pi T_i\in\boxed{\cT}$, one has that $i(\Pi T_i)$ has a final object by convention.  Thus, the last condition of the above lemma is satisfied, and we need only concern ourselves with the first two.

Now, we will introduce our augmented functorial precylinder, which we will denote $(I, \partial^0, \partial^1)$.  In particular, $I=[1]\times (-)$ and $\partial^i$ is the inclusion of either of the endpoint copies of what we start with via \[(-)\overset{\sim}{\to}[0]\times(-)\rightrightarrows [1]\times(-).\] 

Now, let us go through the remaining two points.  First, note that $\Pi T_i\overset{\partial^0_{\Pi T_i}}{\hookrightarrow}[1]\times\Pi T_i$ is a sieve and that $\Pi T_i\overset{\partial^1_{\Pi T_i}}{\hookrightarrow}[1]\times\Pi T_i$ is its complementary cosieve.  Thus, the first point of the above proposition is automatically satisfied.  The last point follows from noting that if we have $f:\Pi S_j\to \Pi T_i$, the following square commutes:
\[
\begin{tikzcd}
  \Pi S_j \arrow[r, "f"] \arrow[d, hook, "\partial^1_{\Pi S_j}"]
    & \Pi T_i \arrow[d, hook, "\partial^1_{\Pi T_i}"] \\
  \left [1\right ]\times\Pi S_j \arrow[r, "\id_{\left [1\right ]}\times f"]
& \left [1\right ]\times \Pi T_i \end{tikzcd}.
\]
Thus, we have that $\boxed{\cT}$ is a local test category.

\end{proof}

Thus we have

\begin{thm}
$\boxed{\cT}$ is a test category.
\end{thm}
\begin{proof}
Since $\boxed{\cT}$ is a local test category and acyclic, it is a test category.
\end{proof}

\end{document}